\documentclass{amsart}

\usepackage{amsmath}
\usepackage{amssymb}
\usepackage{amsthm}
\usepackage{amsfonts}
\usepackage{hyperref}
\usepackage{mathrsfs}

\numberwithin{equation}{section} 

\newtheorem{theorem}{Theorem}[section]

\newtheorem{proposition}[theorem]{Proposition}
\newtheorem{lemma}[theorem]{Lemma}

\newtheorem{definition}[theorem]{Definition}
\newtheorem{remark}[theorem]{Remark}
\newtheorem{example}[theorem]{Example}

\newtheorem{question}[theorem]{Question}
\usepackage{ragged2e}
\usepackage{etoolbox}
\usepackage{lipsum}

\apptocmd{\frame}{}{\justifying}{}
\begin{document}

\title{Hankel Determinantal Ring have $F$-regular Singularities}


\author{Jyoti Singh}
\address{Department of Mathematics, Visvesvaraya National Institute of Technology, Nagpur, Maharashtra-440010, India}
\curraddr{}
\email{jyotijagrati@gmail.com}
\thanks{}

\author{Nikhil P. Zade}
\address{Department of Mathematics, Visvesvaraya National Institute of Technology, Nagpur, Maharashtra-440010, India}
\curraddr{}
\email{zadenikhil8@gmail.com}
\thanks{}



\date{}

\dedicatory{}

\begin{abstract}
In this article, we study Hankel determinantal rings, a special class of determinantal rings defined by minors of Hankel matrices of indeterminates. We discuss the question posed in \cite[Question~4.8]{conca2018hankel} by showing that the test ideal of a Hankel determinantal ring coincides with the ring itself.
\end{abstract}
\maketitle

\section{Introduction}
Singularities of algebraic varieties in positive characteristic have been
studied extensively using the Frobenius endomorphism, the $p$-th power map
$F \colon R \to R, \ r \mapsto r^p$, which is a ring homomorphism whenever the ring has
characteristic $p > 0$. A foundational result of
Kunz~\cite{kunz1969characterizations} shows that a Noetherian scheme over a
field of characteristic $p > 0$ is regular if and only if the Frobenius is
flat. This result opened the door to a systematic study of singularities via
the Frobenius, leading to the theory of $F$-singularities. The main classes
in this theory --- $F$-pure, $F$-rational, $F$-regular and strongly $F$-regular rings --- are
defined by how the Frobenius acts on ideals and modules, and they serve as
characteristic-$p$ counterparts of well-known singularity types in complex
geometry~\cite{hara2000f, doi:10.1080/00927870008827196}.     
For $F$-finite rings, these classes 
satisfy the chain of implications:
$$\text{strongly } F\text{-regular}  \Rightarrow   F\text{-regular}  \Rightarrow    F\text{-rational}, $$
and under the Gorenstein assumption, $F$-rational implies strongly $F$-regular. For further
background and developments in this area, we refer
to~\cite{hochster1976purity, goto1977structure, Fedder1983FpurityAR,
ff900e0e-2e54-32a4-950a-f7c375e78bfd, smith1997f, hara2003generalization, schwede2012survey}.

Tight closure theory, developed by Hochster and
Huneke~\cite{86e48985-411d-388b-82fc-601614430504, hochster2022tight}, provided the algebraic
framework for the theory of $F$-singularities. A central invariant arising from
this theory is the \emph{test ideal} $\tau(R)$, which measures the severity of
singularities of a ring $R$ in characteristic $p > 0$. Test ideals serve as
characteristic-$p$ analogues of multiplier ideals in complex birational geometry,
and the two invariants are closely related: for a variety defined over a number
field, reduction modulo $p$ sends the multiplier ideal to the test ideal for
almost all primes $p$~\cite{Hara2001GeometricIO,
doi:10.1080/00927870008827196}.

A key geometric role of the test ideal is to identify the non-$F$-regular locus
of $\operatorname{Spec}(R)$: the support of $\tau(R)$ is precisely the set of
points $\mathfrak{p} \in \operatorname{Spec}(R)$ at which the localisation $R_\mathfrak{p}$
fails to be strongly $F$-regular~\cite{ schwede2012survey}.
When $R$ is strongly $F$-regular,
$\tau(R) = R$, reflecting the absence of bad singularities; as the singularities
worsen, $\tau(R)$ cuts out a larger and larger closed subset of $\operatorname{Spec}(R)$.
Test ideals are also closely connected to the theory of Frobenius splittings: the
uniformly $F$-compatible ideals, which are precisely the ideals preserved by all
Cartier maps on $R$, determine $\tau(R)$ as their smallest nonzero
member~\cite{98751175-d0fb-3a6c-bf81-2c1d043cc7aa, Schwede_2009}, and they
correspond geometrically to the compatibly split subschemes of $\operatorname{Spec}(R)$.

In this article, we study the test ideal using the tools of commutative algebra.
We view the test ideal as a distinguished ideal in the lattice of ideals that are
compatible with the Frobenius map. Our approach follows Schwede~\cite{Schwede_2009}
rather than the original definition of Hochster and
Huneke~\cite{86e48985-411d-388b-82fc-601614430504, hochster2022tight}: instead of using tight closure,
we define $\tau(R)$ as the smallest nonzero ideal preserved by all Cartier maps on
$R$. This perspective is more direct and connects naturally with the theory of
Frobenius splittings due to Mehta and
Ramanathan~\cite{ff900e0e-2e54-32a4-950a-f7c375e78bfd} (see
also~\cite{Ramanan1985}).

The main focus of this article is the study of $F$-regular singularities of Hankel determinantal rings. Recall that a Hankel determinantal ring is a quotient of a
polynomial ring by the ideal generated by minors of a Hankel matrix, a matrix whose entries are constant along each anti-diagonal. Determinantal rings form a
broad and well-studied family in commutative algebra; besides Hankel determinantal rings, this family includes generic determinantal rings, symmetric determinantal
rings, and Pfaffian rings. Generic determinantal rings and Pfaffian rings are known to be strongly $F$-regular; see \cite{hochster1994tight, conca1997ladder, bruns1998f, baetica2001f, de2024blowup}. If $n\equiv t \operatorname{mod}(2)$, then symmetric determinantal rings  are Gorenstein, where $n$ is order of symmetric matrix  and $1\leq t\leq n $ is order of minor, see \cite{goto1979gorensteinness}, and it is $F$-regular, see \cite{conca2025invariant}.

Conca et al.~\cite{conca2018hankel} showed that Hankel determinantal rings over a field of characteristic zero have rational singularities. They also proved that, in positive characteristic $p$, Hankel determinantal rings with respect to maximal minors are $F$-rational whenever $p$ is greater than or equal to the size of the maximal minor, and are $F$-pure in general; see \cite[Theorems~2.1 and~4.1]{conca2018hankel}. Furthermore, they raised the question of whether Hankel determinantal rings in positive characteristic are  $F$-regular.  
The present article makes progress on this question by proving the strong 
$F$-regularity of Hankel determinantal rings defined by the maximal minors of 
a square Hankel matrix, that is, by the determinant of an $n \times n$ Hankel 
matrix of indeterminates. Our approach is based on Schwede's criterion \cite[Theorem~3.19]{schwede2012survey}, which states that an integral domain essentially of finite type over a perfect field is strongly $F$-regular if and only if its test ideal coincides with the ring itself. Since Hankel determinantal rings are domains by \cite{watanabe1997hankel}, this criterion applies in our setting.

Motivated by this criterion, we now state the main result of this article;

\begin{theorem}\label{4}
Let $S=K[x_1,\dots, x_{2n-1}]$ be a polynomial ring over a perfect field $K$ of positive characteristic, and let $f$ be the determinant of $n\times n$ Hankel matrix $H$ whose $ij$-th entries are $x_{i+j-1}$. Then the test ideal of $R:=S/(f)$ is ring $R$ itself.
 \end{theorem}

This paper is organized as follows. In Section \ref{S1}, we study the $R$-module $F_*^e(R)$ induced by the Frobenius map $F^e$, recall the definition of the test ideal via compatible ideals, and discuss its connection with strong $F$-regularity. In Section \ref{S3}, we introduce Hankel determinantal rings, compute their test ideals, and discuss that these rings are strongly $F$-regular.


\section{Preliminaries}\label{S1}
Throughout this article, all rings are commutative Noetherian rings with identity
and of prime characteristic $p > 0$. We fix $K$ to denote an algebraically closed
field of characteristic $p$.

\subsection*{The Frobenius Endomorphism}

Let $R$ be a ring of characteristic $p > 0$. The \emph{Frobenius endomorphism} is
the map
\[
  F \colon R \longrightarrow R, \qquad r \longmapsto r^p.
\]
Since $(r + s)^p = r^p + s^p$ and $(rs)^p = r^p s^p$ for all $r, s \in R$, the
map $F$ is a ring homomorphism. Its image $R^p := F(R) = \{r^p \mid r \in R\}$ is
a subring of $R$. We record the standard criterion for injectivity:

\begin{remark}
The Frobenius endomorphism $F \colon R \to R$ is injective if and only if $R$ is
reduced.
\end{remark}

Iterating $F$ yields, for each $e \in \mathbb{N}$, the \emph{$e$-th iterated
Frobenius}
\[
  F^e \colon R \longrightarrow R, \qquad r \longmapsto r^{p^e},
\]
whose images form an infinite descending chain of subrings:
\[
  R \supseteq F(R) \supseteq F^2(R) \supseteq \cdots.
\]

\begin{definition}\label{D111}
Let $I \subseteq R$ be an ideal and $e \in \mathbb{N}$. The \emph{$p^e$-th
Frobenius power} of $I$ is the ideal
\[
  I^{[p^e]} := \bigl(\{a^{p^e} \mid a \in I\}\bigr).
\]
\end{definition}

For $t \in \mathbb{N}$, we write $I^{t[p^e]}$ for the ideal $(I^t)^{[p^e]}$,
which coincides with $(I^{[p^e]})^t$.

\subsection*{The Module $F^e_*M$ and $F$-Finiteness}

Given an $R$-module $M$, we write $F^e_*M$ for the $R$-module obtained from $M$
by \emph{restriction of scalars} along $F^e$. Concretely, $F^e_*M$ coincides
with $M$ as an abelian group, but the $R$-action is twisted: for $r \in R$ and
$m \in M$,
\[
  r \cdot F^e_*m := F^e_*r^{p^e}m.
\]
In particular, $F^e_*R$ is an $R$-module via this restricted scalar action.

\begin{definition}
A ring $R$ of characteristic $p > 0$ is \textbf{$F$-finite} if $F^e_*R$ is a
finitely generated $R$-module for every $e > 0$.
\end{definition}

For a general ring $R$, the $R$-module $\operatorname{Hom}_R(F^e_*R, R)$ can be
difficult to compute. For polynomial rings, however, a clean description is
available.

\begin{lemma}[{\cite[Remark~3.2]{enescu2020strong}}]\label{L1}
Let $S = K[x_1, \dots, x_n]$ be a positively graded polynomial ring, and set
\[
  \Lambda := \{(l_1, \dots, l_n) \in \mathbb{N}^n \mid 0 \le l_i < p^e
  \text{ for all } i\}.
\]
The set $\beta := \{F^e_*x_1^{l_1} \cdots x_n^{l_n} \mid l \in \Lambda\}$ is a
free $S$-basis for $F^e_*S$. For each $l \in \Lambda$, define the $S$-linear
map $\phi_l \colon F^e_*S \to S$ by
\[
  \phi_l(F^e_*x_1^{l_1'} \cdots x_n^{l_n'}) :=
  \begin{cases}
    1 & \text{if } l = l', \\
    0 & \text{otherwise.}
  \end{cases}
\]
Setting $\Phi_e := \phi_{(p^e - 1, \dots, p^e - 1)}$, the module
$\operatorname{Hom}_S(F^e_*S, S)$ is a cyclic $F^e_*S$-module generated by
$\Phi_e$, which is called the \emph{trace map}.
\end{lemma}

The following proposition extends this description to quotient rings.

\begin{proposition}[{\cite[Theorem~3.3]{enescu2020strong}}]\label{P1}
Let $S := K[x_1, \dots, x_n]$ be a positively graded polynomial ring and let $I \subseteq S$ be a homogeneous ideal with
$R := S/I$. There is an isomorphism of $F^e_*R$-modules
\[
  \frac{I F^e_*S :_{F^e_*S} F^e_*(I)}{I F^e_*S}
  \;\xrightarrow{\;\sim\;}
  \operatorname{Hom}_R(F^e_*R, R),
\]
under which $F^e_*s$ maps to the $R$-linear map $[F^e_*s \cdot \Phi_e]
= \Phi_e(F^e_*s\,{-})$.
\end{proposition}

For principal ideals, the colon ideal in Proposition~\ref{P1} has a particularly
explicit form.

\begin{remark}\label{R1}
Let $S := K[x_1, \dots, x_n]$ and suppose $I = (f) \subseteq \mathfrak{m}
= (x_1, \dots, x_n)$, where $f$ is any homogeneous polynomial of $S$. Then
\[
  I F^e_*S :_{F^e_*S} F^e_*I = F^e_*f^{p^e - 1} \cdot F^e_*S.
\]
Consequently, if $R := S/I$, every $R$-linear map $\phi \colon F^e_*R \to R$
has the form $\phi = \Phi_e(F^e_*f^{p^e - 1} g \cdot \,{-})$ for some
$F^e_*g \in F^e_*S$, and such a map is zero if and only if
$F^e_*g $ is multiple of $ F^e_*f$.
\end{remark}

\subsection*{Compatible Ideals and the Test Ideal}
The test ideal of a reduced Noetherian $F$-finite ring can be characterized as the smallest positive height ideal uniformly compatible with all maps in $\operatorname{Hom}_R(F_*^eR,R)$ for every $e>0$. For further background, we refer the reader to \cite{Schwede_2009, CARVAJALROJAS202025}.

\begin{definition}\label{D11}
An ideal $I \subseteq R$ is called \emph{$\phi$-compatible} for an $R$-linear map
$\phi \colon F^e_*R \to R$ if $\phi(F^e_*I) \subseteq I$.
\end{definition}

\begin{remark}
An ideal $I$ is $\phi$-compatible if and only if $\phi$ descends to an $R$-linear
map $F^e_*R/I \to R/I$.
\end{remark}

\begin{definition}
An ideal $I \subseteq R$ is \emph{uniformly $F$-compatible} if it is
$\phi$-compatible for every $R$-linear map $\phi \in \operatorname{Hom}_R(F^e_*R,
R)$ and every $e > 0$.
\end{definition}

\begin{definition}[Test ideal]\label{D1}
Let $R$ be an $F$-finite Noetherian reduced ring. The \emph{test ideal}
$\tau(R)$ is the smallest nonzero uniformly $F$-compatible ideal of $R$ that is
not contained in any minimal prime of $R$.
\end{definition}

\subsection*{Strongly $F$-Regular Rings}

\begin{definition}[Strongly $F$-regular ring {\cite{hochster2022tight}}]\label{D:SFR}
Let $R$ be a reduced $F$-finite Noetherian ring of characteristic $p > 0$, and set
$R^\circ := R \setminus \bigcup_{\mathfrak{p} \in \operatorname{Min}(R)}
\mathfrak{p}$. The ring $R$ is \emph{strongly $F$-regular} if for every
$c \in R^\circ$ there exists $e \geq 0$ such that the $R$-linear map
$R \to F^e_* R$, $1 \mapsto F^e_* c$, splits; equivalently, there exists
$\phi \in \operatorname{Hom}_R(F^e_* R, R)$ with $\phi(F^e_* c) = 1$.
\end{definition}

The following connection between the test ideal and strong $F$-regularity is fundamental:

\begin{theorem}[{\cite[Theorem~3.19]{schwede2012survey}}]\label{1}
Let $R$ be a domain essentially of finite type over a perfect field $K$. Then
$\tau(R) = R$ if and only if $R$ is strongly $F$-regular.
\end{theorem}


\section{$F$-regularity of The Hankel Determinantal Ring}\label{S3}
In this section, we study Hankel determinantal rings and show that, for certain cases, their test ideals coincide with the rings themselves. Consequently, we shows the strong $F$-regularity of these rings.

A \textit{Hankel matrix} is a matrix of the form:
\[
H = \begin{pmatrix} 
x_1 & x_2 & x_3 & \cdots & x_n \\ 
x_2 & x_3 & \cdots & \cdots & x_{n+1} \\ 
x_3 & \cdots & \cdots & \cdots & x_{n+2} \\ 
\vdots & & & & \vdots \\ 
x_m & \cdots & \cdots & \cdots & x_{n+m-1} 
\end{pmatrix},
\]
with $m\leq n$, where $x_1, \dots, x_{n+m-1}$ are indeterminates over any field $\mathbb{F}$. That is $H$ is a $m \times n$ matrix whose $(i,j)$-entry is $x_{i+j-1}$. The 
key feature is that $H$ is constant along each anti-diagonal. The 
\textit{Hankel determinantal ring} is defined as:
\[
R = \mathbb{F}[x_1, \dots, x_{n+m-1}]/I_t(H),
\]
where $I_t(H)$ is the ideal generated by all $t \times t$ minors of $H$, 
and $1 \leq t \leq m$.

Hankel determinantal rings over a field of positive characteristic $p$ are known to be $F$-pure and, in the case of maximal minors, $F$-rational whenever $p\geq m$; see \cite[Theorems~2.1 and~4.1]{conca2018hankel}. Consequently, for $p\geq t=m=n$, the corresponding Hankel determinantal ring is strongly $F$-regular, since hypersurfaces are Gorenstein. In this section, we discuss the strong $F$-regularity of Hankel determinantal rings over general positive characteristics $p$.
\begin{remark}
The Hankel determinantal rings are normal domains \cite{watanabe1997hankel}. Hence, by Theorem \ref{1}, a Hankel determinantal ring $R$ is strongly $F$-regular if and only if $\tau(R)=R$.
\end{remark}

Motivated by this observation, we address \cite[Question~4.8]{conca2018hankel} by proving that the test ideal of a Hankel determinantal ring coincides with the ring itself. In particular, we prove the strong $F$-regularity of these rings.

The following theorem proves this result in the case $t=m=n$ over any perfect field of positive characteristic. In the proof of following theorem, we use the notation $$H^{p^e-1}=x_1^{p^e-1}x_2^{p^e-1}\cdots x_{2n-1}^{p^e-1}$$ and hence $\Phi_e(F_*^eH^{p^e-1})=1\in S$.

\begin{theorem}\label{3}
Let $S$ be the polynomial ring $K[x_1,\dots,x_n, \dots,x_{2n-1}]$ over a perfect field $K$ of positive characteristic $p$, with  $n\times n$  Hankel matrix $H$.
Let $f$ be the  determinant of $H$. Then the test ideal of quotient ring $R:=S/(f)$ is ring $R$ itself.
\end{theorem}
\begin{proof} 
By Remark \ref{R1}, every $R$-linear map $\phi \in \operatorname{Hom}_R(F_*^eR, R)$ is of the form
$\Phi_e(F_*^ef^{p^e-1}g\cdot -)$ for some $g \in S$. Hence, by the definition of the test ideal, for every $a \in \tau(R)$ we have
$\Phi_e(F_*^ef^{p^e-1}gF_*^ea)=\Phi_e(F_*^ef^{p^e-1}ga)\in \tau(R)$. We will show that there exists an element $F_*^ea \in F_*^e\tau(R)\setminus F_*^e(f)$ whose image under a suitable composition of $R$-linear $\Phi_e(F_*^ef^{p^e-1}g\cdot -)$ is a unit.

Let $a \in \tau(R)\subseteq R$ and write
$a=\sum_{\gamma} a_{\gamma}H^{\gamma},$
where $\gamma=(\gamma_1,\dots,\gamma_{2n-1})\in \mathbb{N}^{2n-1}$, $H^{\gamma}=x_1^{\gamma_1}\cdots x_{2n-1}^{\gamma_{2n-1}}$, and $a_{\gamma}\in K$. We may assume that the exponent of $x_n$ in every term of $a$ is strictly less than $n$; otherwise, replacing $a$ modulo $f$ yields such an expression, since $f$ contains the term  either positive or negative of $ x_n^n$. Since we are only concerned with determining whether elements  belong to the ideal $\tau(R)$, we may ignore the signs of $x_n^n$ in $f$ and assume it positive. Similarly, write
$f^{p^e-1}=\sum_{\alpha} c_{\alpha}H^{\alpha},$
where $\alpha=(\alpha_1,\dots,\alpha_{2n-1})\in \mathbb{N}^{2n-1}$ satisfies
$\alpha_1+\cdots+\alpha_{2n-1}=n(p^e-1)$,
$H^{\alpha}=x_1^{\alpha_1}\cdots x_{2n-1}^{\alpha_{2n-1}}$, and $c_{\alpha}$ is some integer. In particular, $f^{p^e-1}$ contains the monomials
$x_1^{p^e-1}x_3^{p^e-1}\cdots x_{2n-1}^{p^e-1}$
and $x_n^{n(p^e-1)}$.

Assume that $c_{\gamma}H^{\gamma}$ is a term of $a$ whose $x_n$-exponent is maximal among all terms of $a$, and is strictly less than $n$. Choose $e>0$ sufficiently large so that $p^e-1>\deg(a)$. Set
$$g=\dfrac{H^{p^e-1}}{x_n^{p^e-n}H^{\gamma}}\in S.$$
Then
$\Phi_e(F_*^ef^{p^e-1}ga)
=
\sum_{\alpha}c_{\alpha}\sum_{\gamma} b_{\gamma}\Phi_e(F_*^eH^{\alpha}gH^{\gamma}),$
where $b_{\gamma}^{p^e}=a_{\gamma}$. Since $f^{p^e-1}$ contains the term $x_n^{n(p^e-1)}$, we obtain
$\sum_{\gamma} b_{\gamma}\Phi_e(F_*^ex_n^{n(p^e-1)}gH^{\gamma})
=
b_{\gamma}x_n^{n-1},$
because for every other term $b_{\gamma'}H^{\gamma'}\neq b_{\gamma}H^{\gamma}$ of $a$,
$\Phi_e(F_*^ex_n^{n(p^e-1)}gH^{\gamma'})=0.$
Moreover, for every term $c_{\alpha}H^{\alpha}\neq x_n^{n(p^e-1)}$ of $f^{p^e-1}$, the expression
$c_{\alpha}\sum_{\gamma} b_{\gamma}\Phi_e(F_*^eH^{\alpha}gH^{\gamma})$
is either zero or a constant multiple of a monomial of degree $n-1$ different from $b_{\gamma}x_n^{n-1}$. Hence,
$\Phi_e(F_*^ef^{p^e-1}ga)=b_{\gamma}x_n^{n-1}+b_1\in \tau(R),$
where $b_1$ is either zero or a homogeneous polynomial of degree $n-1$ such that every term of $b_1$ has $x_n$-exponent strictly less than $n-1$. Furthermore, the exponents of $x_1$ and $x_{2n-1}$ in each term of $b_1$ are at most one, since the exponents of $x_1$ and $x_{2n-1}$ in $H^{\alpha}H^{\gamma}$ do not exceed $2(p^e-1)$.

 If $n=2$, then $b_1=0$. Indeed, for $\Phi_e(F_*^eH^{\alpha}ga)$ to be nonzero, we must have $\alpha_2\geq p^e-2$, since every term of $ga$ has $x_2$-exponent at most one. Consequently,
$\alpha_1,\alpha_3\leq p^e-1-\frac{p^e-2}{2}=\frac{p^e}{2}$,
which is impossible for sufficiently large $e$. Hence,
$\Phi_e(F_*^ef^{p^e-1}ga)=b_{\gamma}x_2\in \tau(R),$
and therefore $x_2\in \tau(R)$. Now, taking
$$g=\frac{H^{p^e-1}}{x_1^{p^e-1}x_3^{p^e-1}x_2}\in S,$$
we obtain
$\Phi_e(F_*^ef^{p^e-1}gx_2)=1,$
which implies that $1\in \tau(R)$.

If $n\geq 3$, then we have
$x_n^{n-1}+b_1:=x_n^{n-1}+b_{\gamma}^{-1}b_1\in \tau(R).$
Consider
$$g=\frac{H^{p^e-1}x_n^{p^e}}{x_1^{p^e-1}x_{2n-1}^{p^e-1}x_n^{p^e-n+2}x_n^{n-1}}\in S.$$
Then
$$\Phi_e(F_*^ef^{p^e-1}g(x_n^{n-1}+b_1))
=
x_n^{n-2}+\Phi_e(F_*^ef^{p^e-1}gb_1),$$
which yields
$x_n^{n-2}+b_2\in \tau(R),$
where $b_2$ is either zero or a homogeneous polynomial of degree $n-2$ such that every term of $b_2$ has $x_n$-exponent strictly  less than $n-2$, and the exponents of $x_1,x_2,x_{2n-2},x_{2n-1}$ are zero. 
Because, every term of $gb_1$ has $x_1$- and $x_{2n-1}$-exponents at most one. Because, for $\Phi_e(F_*^ef^{p^e-1}gb_1)$ to be nonzero, the corresponding exponents $\alpha_1$ and $\alpha_{2n-1}$ must be at least $p^e-2$.

If $n\geq 5$, consider
$$g=\frac{H^{p^e-1}x_n^{p^e}}{x_1^{p^e-1}x_3^{p^e-1}x_{2n-3}^{p^e-1}x_{2n-1}^{p^e-1}x_n^{p^e-n+4}x_n^{n-2}}\in S.$$
Arguing as above, $\Phi_e(F_*^ef^{p^e-1}g(x_n^{n-2}+b_2))$ yields
$x_n^{n-4}+b_3\in \tau(R),$
where $b_3$ is either zero or a homogeneous polynomial of degree $n-4$ such that every term of $b_3$ has $x_n$-exponent strictly less than $n-4$, and the exponents of
$x_1,x_2,x_3,x_4,x_{2n-4},x_{2n-3},x_{2n-2},x_{2n-1}$
are zero.
Continuing in this manner, we conclude that if $n$ is odd, then $x_n\in \tau(R)$, while if $n$ is even, then
$x_n^2+b\in \tau(R),$
where $b$ is either zero or a homogeneous polynomial of degree $2$ such that every term of $b$ has $x_n$-exponent strictly less than $2$, and the exponents of
$x_1,\dots,x_{n-2},x_{n+2},\dots,x_{2n-1}$
are zero.

If $n$ is odd, then $x_n\in \tau(R)$. Consider
$$g=\frac{H^{p^e-1}}{x_1^{p^e-1}x_3^{p^e-1}\cdots x_{n-2}^{p^e-1}x_{n+1}^{p^e-2}x_{n+4}^{p^e-1}x_{n+6}^{p^e-1}\cdots x_{2n-1}^{p^e-1}x_n}\in S.$$
(For example, when $n=3$ and $n=5$,
$$g=\frac{H^{p^e-1}}{x_1^{p^e-1}x_4^{p^e-2}x_3}
\quad \text{and} \quad
g=\frac{H^{p^e-1}}{x_1^{p^e-1}x_3^{p^e-1}x_6^{p^e-2}x_9^{p^e-1}x_5} \in S \quad \text{respectivly}.)$$
Then
$\Phi_e(F_*^ef^{p^e-1}gx_n)$
yields
$x_{n+1}\in \tau(R)$
(and $x_4,x_6\in \tau(R)$ in the above examples).
Now, taking
$$g=\frac{H^{p^e-1}}{x_1^{p^e-1}x_3^{p^e-1}\cdots x_{2n-1}^{p^e-1}x_{n+1}}\in S,$$
$\Phi_e(F_*^ef^{p^e-1}gx_{n+1})$
yields  $1\in \tau(R)$.

If $n$ is even, then $x_n^2+b\in \tau(R)$. Consider
$$g=\frac{H^{p^e-1}}{x_1^{p^e-1}x_3^{p^e-1}\cdots x_{n-1}^{p^e-1}x_{n+2}^{p^e-2}x_{n+5}^{p^e-1}x_{n+7}^{p^e-1}\cdots x_{2n-1}^{p^e-1}x_n^2}\in S.$$
Then
$\Phi_e(F_*^ef^{p^e-1}g(x_n^2+b))$
yields
$x_{n+2}\in \tau(R)$,
since
$\Phi_e(F_*^ef^{p^e-1}gb)=0.$ Indeed, for $\Phi_e(F_*^ef^{p^e-1}gb)$ to be nonzero, we must have
$\alpha_1,\alpha_3,\dots,\alpha_{n-3},\alpha_{n+5},\alpha_{n+7},\dots,\alpha_{2n-1}=p^e-1,$
$\alpha_{n+3}=0,$
and $\alpha_{n-1}\geq p^e-2$, which forces
$\alpha_{n+2}\geq 2(p^e-2).$
However, $\alpha_{n+2}$ can only be either $2p^e-2$ or $p^e-2$ for nonzero. Now, taking
$$g=\frac{H^{p^e-1}}{x_1^{p^e-1}x_3^{p^e-1}\cdots x_{2n-1}^{p^e-1}x_{n+2}}\in S,$$
$\Phi_e(F_*^ef^{p^e-1}gx_{n+2})$
yields $1\in \tau(R)$.
\end{proof}

Note that the proof of the above theorem relies on the description of the $R$-linear maps in $\operatorname{Hom}_R(F_*^eR,R)$, which correspond to the elements of the ideal
$$\frac{I_tF_*^eS:_{F_*^eS}F_*^eI_t}{I_tF_*^eS}$$ of $F_*^eS$.
The computation of this ideal becomes significantly more difficult when $I_t$ is not principal. Nevertheless, an explicit description of this ideal
would lead to a computation of the test ideal of $R=K[H]/I_t$.

\begin{remark}\label{5}
A Hankel determinantal ring $R$ defined by a $2\times n$ Hankel matrix over a field $\mathbb{F}$ is naturally isomorphic to the $n$-th Veronese subring of the polynomial ring $\mathbb{F}[u,v]$; see \cite[Section~2]{conca2018hankel}. In particular, $R$ is a direct summand of $\mathbb{F}[u,v]$. Therefore, if $\mathbb{F}$ has positive characteristic, then $R$ is strongly $F$-regular by \cite[Theorem~3.1(e)]{hochster1989tight}; see also \cite[Theorem~9.7]{hochster2022tight}.
\end{remark}

The following example gives the test ideal of some Hankel determinantal ring other than the condition $t=m=n$, computed using Macaulay2.
 
\begin{example}
Consider the polynomial ring
$S=\mathbb{Z}/2\mathbb{Z}[u,v,w,x,y,z]$
and the Hankel matrix
$$H=
\begin{pmatrix}
u & v & w & x\\
v & w & x & y\\
w & x & y & z
\end{pmatrix}.$$
Let
$I$ be the ideal generated by all $3\times 3$ minors of the matrix $H$.
The following Macaulay2 commands compute the test ideal of Hankel determinantal ring $S/I$:
\begin{verbatim}
i1 : S=ZZ/2[x,y,z,w,v];
      I=ideal(u*w*y-u*x^2+v*x*w-v^2*y+w*v*x-w^3, 
              u*x*z-u*y^2+w^2*y-w*v*z+x*v*y-x^2*w, 
              u*w*z-u*x*y+v*y*w-v^2*z+x^2*v-x*w^2, 
              v*x*z-v*y^2+w*y*x-w^2*z+x*w*y-x^3);
      R=S/I

o2 : Ideal of S
o3 =  R
o3 :  QuotientRing

i4 : testIdeal R

o4 =  ideal 1 
o4 :  Ideal of R
\end{verbatim}
Hence, the test ideal of $R$ is the unit ideal, and therefore $R$ is strongly $F$-regular.
\end{example}

Motivated by the preceding example, Remark \ref{5} and Theorem \ref{3}, we pose the following question.
\begin{question}
Is every Hankel dterminantal rings are strongly $F$-regular?
\end{question}\

\noindent \textbf{Acknowledgement:} The authors thank Claudiu Raicu for helpful comments and suggestion about Remark \ref{5}. We use Macaulay2 \cite{M2} to compute an example. The first author is partially supported by National Board for Higher Mathematics, DAE, Govt. of India (Research project no. $02011$/$24$/$2024$/ NBHM(R.P)/R \& D II/$9826$).





\bibliographystyle{amsalpha}
\bibliography{Ref}
\end{document}